\documentclass[11pt, a4paper]{article}
\usepackage{amsmath}
\usepackage{amsthm}
\usepackage{amssymb}
\usepackage{graphicx}
\usepackage[all]{xy}
\usepackage{verbatim}

\newtheorem{theorem}{Theorem}[section]
\newtheorem{lemma}{Lemma}
\newtheorem{prop}{Proposition}
\newtheorem{remark}{Remark}
\newtheorem{corollary}{Corollary}
\newtheorem{definition}{Definition}
\newtheorem{example}{Example}

\newenvironment{potwr}[1]{{\noindent\it {Proof of Theorem \ref{#1}}}.}{\hfill$\square$}

\newcommand{\trans}{\pitchfork}
\newcommand{\R}{\mathbb R}
\newcommand{\Z}{\mathbb Z}

\graphicspath{{../}}

\title{On existence of periodic solutions for Kepler type problems}
\begin{document}
\author{P. Amster and J. Haddad}
\date{}
\maketitle

\begin{abstract}
We prove existence and multiplicity of periodic motions for the forced 2-body problem under conditions of topological character.
In the different cases, the lower bounds obtained for the number of solutions are related to the winding number of a curve in 
the plane, the homology of a space in $\R^3$, the knot type of a curve and the link type of a set of curves.
Also, the results are applied to the restricted $n$-body problem.

{\bf MSC 2010}: 34B15, 34C25, 70K40.

\end{abstract}

\label{ortega1}

\section{Introduction}
Let us consider the 
periodic singular problem

\begin{equation}
\label{equa}
\left\{
\begin{array}{c}
u''(t) \pm \frac {u(t)}{|u(t)|^{q+1}} = \lambda h(t)\\
u(0) = u(1)\\
u'(0) = u'(1)
\end{array}
\right.
\end{equation}
for a vector function $u:I=[0,1]\to \R^n$, where $q \geq 2$ and 
$h\in C(I,\R^n)$ with 
$\overline{h}:=\int_0^1h(t)\, dt = 0$.

Here $u$ describes the motion of a particle under a singular central force that can be attractive or repulsive depending on the sign $\pm$, and an arbitrary perturbation $h$.

In \cite{O1} we studied the case $n=2$, for which we proved 
the existence of periodic solutions 
under a non-degeneracy condition. In more precise terms, we obtained 
a lower bound of the number of solutions that depends purely on a topological property of the second primitive of $h$: 

\begin{theorem}
\label{oldthm}
For $n=2$, let $H$ be a periodic function such that $H'' = -h$ 
and let $r$ be the number of bounded connected components of $\R^2 \setminus Im(H)$.
Then for $\lambda$ large enough, 
problem \eqref{equa} has at least $r$ solutions.
\end{theorem}

In this work, we extend Theorem \ref{oldthm} in several directions. 

In section \ref{repulsive} we consider the repulsive case. 
We obtain at least one extra solution from the direct computation of the Leray-Schauder degree over the set of curves that are bounded away from the origin and from infinity. 
More precisely, we obtain a lower bound 
of the number of solutions that depends not only on the number of connected components 
of $\R^2\setminus Im(H)$ 
but also on the winding number of $H$ with respect to these components.

The main idea is the following:
firstly, we prove that solutions of $\eqref{equa}$ are uniformly bounded for $\lambda \leq \lambda_*$ and that there are no solutions for $\lambda = 0$.
Thus, if we picture the solution set $S = \{(u,\lambda) \in C^2([0,1],\R^2) \times \R/ u \hbox{ is a solution of } \eqref{equa} \}$ 
then $S$ contains a continuum which starts at a solution $(u_*, \lambda_*)$ given by 
Theorem \ref{oldthm} and is bounded both in the $\lambda$ and the $u$ directions.
Then $S$ must `turn around' in the $\lambda$ direction and intersect again the subspace $\lambda = \lambda_*$.

Furthermore, in section \ref{genericity} we prove that for a `generic' forcing term $h$ the repulsive problem has in fact at least $2r$ periodic solutions.
More specifically, take $\tilde{C}^0_{per}$ the Banach space of continuous periodic functions of zero average. 
Then there exists a residual set $\Sigma \in \tilde{C}^0_{per}$ such that if $\lambda h \in \Sigma$ then all solutions of \eqref{equa} are non-degenerate. 
As a consequence, all of them have multiplicity one and depend differentiably on $h$.

In section \ref{counterexamples} we give some examples illustrating existence and non existence of solutions in some particular situations. 

In section \ref{higherdimensions} we extend Theorem \ref{oldthm} to higher dimensions.
Our proofs make use of some classical results of algebraic topology.
The case $n = 3$ is treated separately because the homology of open sets with smooth boundary is simple and easy to understand, while the case $n > 3$ needs more restrictive hypotheses.

In section \ref{morseandknots} we obtain further results for the case $n=3$, assuming that $H$ is an embedded knot.
The lower bounds for the number of solutions will depend on the knot type of $H$, specifically on a knot invariant called the tunnel number $t(H)$. For example, we prove existence of at least $3$ solutions if $H$ is a nontrivial knot and at least $5$ solutions when it is a composite knot.

Finally, in section \ref{restrictednbody} we apply the methods of section \ref{morseandknots} to the restricted $N$-body problem.

\subsection{Preliminaries}

Theorem \ref{oldthm} is proved 
in \cite{O1} using a result contained in Cronin's book \cite{Cr} about the averaging method.
However, for our purposes it shall be convenient to describe the procedure in a precise way.
As in \cite{O1}, let us make the change of variables
\[u(t) = \lambda (z(t) - H(t))\]
so equation \eqref{equa} becomes

\begin{equation}
\label{equa2}
\left\{
\begin{array}{c}
z''(t) = \mp \epsilon \frac { z(t)- H(t)}{|z(t)-H(t)|^{q+1}}\\
z(0) = z(1)\\
z'(0) = z'(1)
\end{array}
\right.
\end{equation}
where 
\begin{equation}
\label{epslam}
\epsilon = \lambda^{-(q+1)}.
\end{equation}

We shall associate a functional between Banach spaces 
\[\mathcal G_\epsilon:C^2_{per}(I,\R^n) \to C^2_{per}(I,\R^n)\]
to this system,  
which is continuous (in fact, analytic) and such that solutions of 
\eqref{equa2}
are characterized as points $z \in C^2_{per}(I,\R^n)$ which are zeros of $\mathcal G_\epsilon$.
We shall use the Leray-Schauder degree
in order to prove 
that the zeros of $\mathcal G_0$ can be 
continued to zeros of $\mathcal G_\epsilon$ for small values of $\epsilon$. 
Then, it is enough to study the function $\mathcal G_0$, which can be identified with its restriction to the finite dimensional space of constant functions, namely the function $F: \R^n \setminus Im(H) \to \R^n$ given by 
\begin{equation}
\label{funcF}
F(x) = \int_0^1 \frac {x - H(t)}{|x - H(t)|^{q+1}} dt.
\end{equation}

Specifically, for each open set 
$D \subseteq \R^n \setminus Im(H)$ we may associate the open set of curves 
$\mathcal D = \{z \in C^2_{per}(I,\R^n) / Im(z) \subseteq D\}$ and hence
$deg(\mathcal G_0,\mathcal D, 0) = deg(F,D,0)$; thus, it suffices to
look for open sets $D$ such that the latter 
degree is different from zero. 
For each of these sets 
$D$ there exists a solution of \eqref{equa2} for $\epsilon > 0$ small.
In the situation of 
Theorem \ref{oldthm}, if $\Omega_1, \ldots, \Omega_r$ are the bounded connected components of 
$\R^2 \setminus Im(H)$, we may construct as in \cite{O1} open sets $\Omega_i^* \subset \overline{\Omega_i^*} \subset \Omega_i$ where all the respective degrees are equal to $1$.

Moreover, there is also a functional $\mathcal F$ associated to equation \eqref{equa} 
such that
\begin{equation}
\label{GF} \mathcal G_\epsilon (x) = \frac 1 \lambda \mathcal F(\lambda (x - H)) + H
\end{equation}
so it is conjugated to $\mathcal G_\epsilon$ by affine homeomorphisms.
The functional $\mathcal F$ is independent of $\lambda$. It 
shall be constructed explicitly in section \ref{repulsive} and then we shall define $\mathcal G_\epsilon$ by formula \eqref{GF}.

The relation between these two functionals is the following:
the function $u$ is a solution of \eqref{equa} 
if and only if $\mathcal F (u) = - \lambda H$, if and only if $\mathcal G_\epsilon (\frac u \lambda + H ) = \mathcal G_\epsilon (z) = 0$, if and only if $z = \frac u \lambda + H$ is a solution of \eqref{equa2}.
Also, the degrees are related by
\[deg(\mathcal F, \mathcal E, -\lambda H) = deg(\mathcal G_\epsilon, \mathcal D, 0)\]
where $\epsilon$ and $\lambda$ are related by (\ref{epslam}) and $\mathcal D$ and $\mathcal E$ are related by
\begin{equation}
\label{ED}
\mathcal D =  H + \frac 1 \lambda \mathcal E.
\end{equation}

For convenience, let us define the function 
\[g:\R^n \setminus \{0\} \to \R,\; g(u) = \frac 1 {|u|^{q-1}},\]
so $\nabla g(u) = - (q-1) \frac u { |u|^{q+1}}$. 
From now on, $H:S^1 \to \R^n$ shall be a periodic second primitive of $-h$, which is unique up to translations.

\section{The repulsive case}
\label{repulsive}
In this section we shall improve Theorem \ref{oldthm} for the repulsive case. 
In first place, let us prove the existence of an extra solution by a direct degree argument:
\begin{theorem}
\label{mainthmrep}
In the conditions of Theorem \ref{oldthm}, the repulsive case of problem \eqref{equa} admits at least $r+1$ solutions.
\end{theorem}
It is worth noticing, however, that Theorem \ref{mainthmrep}
shall be improved as well at the end of this section  
by 
studying the winding number of $H$ with respect to the connected 
components of $\R^2\setminus Im(H)$.

We shall make use of the following two lemmas, which shall provide 
us {\sl a priori} bounds for the solutions. 
Later on, these bounds shall be used also 
in the proofs of Theorems \ref{mainthmR3} and \ref{mainthmRn} for higher dimensions, 
so the results shall be stated in $\R^n$.
We remark that nothing of this can be 
extended to the attractive case.

\begin{lemma}
\label{boundsolutions}
Given $\lambda_* > 0$ there exist constants $R,r > 0$ such that
\[r \leq u(t) \leq R \;\;\; \forall t \in I\]
for any $u:I \to \R^n$ solution of \eqref{equa} with $\lambda \leq \lambda_*$.
\end{lemma}
\begin{proof}

Let us firstly prove that solutions are bounded away from the origin.
Let $u$ be a solution. We define the energy and compute its derivative
\[E(t) = \frac 1 2 |u'(t)|^2 + \frac 1 {(q-1) |u(t)|^{q-1}}\]
\[ E'(t) = \langle u', u'' \rangle - \frac {\langle u', u\rangle}{|u|^{q+1}} = \langle u', \lambda h \rangle \]
\begin{equation}
\label{difeq1}
|E'(t)| \leq \lambda |u'| |h| \le  
\lambda \|h\|_\infty \sqrt{2E(t)}  \leq C \lambda_* \sqrt{E(t)}.\\
\end{equation}

Now multiply the equation by $u$ and integrate,
\[\int_0^1 \langle u, u''\rangle - \int_0^1 \frac 1 {|u|^{q-1}} = \lambda \int_0^1 \langle h, u \rangle\]
\[-\|u'\|_2^2 - \int_0^1 \frac 1 {|u|^{q-1}} = \lambda \int_0^1 \langle h, u \rangle = \lambda \int_0^1 \langle h, u - \overline{u} \rangle \]

\[\frac 1 2 \|u'\|_2^2 + \int_0^1 \frac {q-2}{q-1} \frac 1 {|u|^{q-1}} + \int_0^1 E = - \lambda \int_0^1 \langle h, u - \overline{u} \rangle \]
\[ \int_0^1 E \leq \lambda \left| \int_0^1 \langle h, u - \overline{u} \rangle \right| \leq \lambda \|h\|_2 \|u - \overline{u}\|_2 \]

\begin{equation}
\label{difeq2}
\int_0^1 E \leq \lambda_* C \|u'\|_2 \leq \lambda_* C \sqrt{ \int_0^1 E }.
\end{equation}
From this inequality, 
a bound for $E(t_0)$ for some $t_0$ is obtained. 
Using \eqref{difeq1}, we get a bound for $E(t)$ for all $t$, and hence a bound for $\frac 1 {|u(t)|^{q-1}}$ depending only on $\lambda_*$.

Next, let us prove that solutions are bounded.
By contradiction, suppose there exists a sequence $\{u_n\}$ such that $\|u_n\|_\infty \to \infty$. 
Then $\|u_n - \overline{u_n}\|_\infty \leq C \|u_n'\|_\infty \leq C \sqrt{\|E\|_\infty} \leq C$ which, in turn, implies that if 
$n$ is large then the image of $u_n$ lies in a half-space. This contradicts the fact that $\overline{g(u_n)} =0$.
\end{proof}

\begin{lemma}
\label{nosolutions}
Problem \eqref{equa} has no solutions for $\lambda=0$. 
\end{lemma}

\begin{proof}
Let $\lambda =0$ and suppose $u$ is a solution. 
Multiply equation \eqref{equa} by $u$ and integrate, then 
\[ -\int_0^1 |u'|^2 = \int_0^1 u'' u = \int_0^1 \frac 1 {|u|^{q-1}},\]
a contradiction.
\end{proof}

\begin{lemma}
Solutions of \eqref{equa} are also bounded in $C^2(I,\mathbb R^n)$.
\end{lemma}
\begin{proof}
We know from Lemma \ref{boundsolutions} that $\|u\|_\infty$ and $\|u'\|_\infty$ are bounded. 
From \eqref{equa}, it follows that $\|u''\|_\infty$ is bounded as well.
\end{proof}

\begin{lemma}
Let $\mathcal F: \mathcal E \subset C^2(I, \mathbb R^n) \to C^2(I, \mathbb R^n)$ 
where $\mathcal F$ is the functional associated to \eqref{equa} and $\mathcal E = \{u\in C^2:r < |u| < R, \|u'\|<C, \|u''\| <C\}$ where $r$, $R$ and $C$ are the 
bounds obtained in the preceding lemmas.

Then $deg(\mathcal F, \mathcal E, -\lambda_* H) = deg(\mathcal G_{\epsilon_*}, \mathcal D, 0) = 0$
where $\epsilon_*$ and $\mathcal D$ are defined as in \eqref{epslam} and \eqref{ED}.
\end{lemma}
\begin{proof}
It follows immediately from the continuation theorem.
\end{proof}

Using this fact, we are able to obtain an extra solution of \eqref{equa}.
\medskip

\begin{potwr}{mainthmrep}
Following the notation and the proof of Theorem 1.1 in \cite{O1}, set 
\[\mathcal A_k := \{z \in C^2(I, \mathbb R^n): Im(z) \subseteq \Omega_k^*, \|z'\|_\infty<C, \|z''\|_\infty <C\}\]
with $\Omega_k^* \subseteq \Omega_k$ such that $deg(\nabla g, \Omega_k^*, 0) = 1$, and take $\mathcal E$ as in the previous lemma.

We know that, for some 
$\epsilon_*$ small enough, 
$deg(\mathcal G_{\epsilon_*}, \mathcal A_k, 0) = 1$ and the problem has a solution in $\mathcal A_k$.

By formula \eqref{ED}, $\mathcal D = \{z/ \frac r {\lambda_*} < |z - H| < \frac R {\lambda_*}\}$, so taking $R$ large enough it is seen that 
$\mathcal A_k \subseteq \mathcal D$ for all $k$.
By the previous lemma, $deg(\mathcal G_{\epsilon_*}, \mathcal D, 0) = 0$.
Defining $\mathcal B = \mathcal D \setminus \overline{\bigcup_{k = 1}^r \mathcal A_k} $ we obtain $deg(\mathcal G_{\epsilon_*}, \mathcal B, 0) = -r$, so there exists at least one more solution in $\mathcal B$,
which is obviously different from the others.
\end{potwr}

\smallskip

In the preceding proof, when $r>1$ it is worthy to observe that, 
although the degree of 
$\mathcal G_{\epsilon_*}$ 
is equal to $-r$ we cannot assert 
the existence of $r$ different extra solutions since 
we are not able to ensure that they are non-degenerate. 
But we are still able to distinguish solutions 
using properties that are invariant under continuation. 

\begin{definition}
Let $\gamma : I \to \R^2$ be a continuous curve and let $x \in \R^2\setminus Im(\gamma)$. 
Let $j(x,\gamma) \in \Z$ be defined as the winding number of $\gamma$ around $x$.
The function $x \mapsto j(x,\gamma)$ is constant in each connected component of $\R^2 \setminus Im(\gamma)$.
Thus, if $\Omega \subseteq \R^2$ is one of these components, we define the winding number $j(\Omega, \gamma)$ of $\gamma$ around $\Omega$.
\end{definition}

\begin{theorem}
\label{mainthmrepimproved}
In the conditions of Theorem \ref{mainthmrep}, 
let $\Omega_1, \ldots, \Omega_r$ be the connected components of $\R^2 \setminus Im(H)$ and let $s$ be the cardinality of the set $\{j(\Omega_k,H)/k = 1, \ldots, r\}$. Then the repulsive case of 
\eqref{equa} admits at least $r+s$ solutions.
\end{theorem}

\begin{proof}
Let $\mathcal E$ be defined as before and 
consider the `winding number function' $J:\mathcal E \to \Z$ defined by $J(x) = j(0,x)$, which 
determines in $\mathcal E$ the connected components $\mathcal E_i =\{u \in \mathcal E/ J(x)=i\}$ for $i \in \Z$.
Since $\partial \mathcal E = \bigcup_{i \in \Z} \partial \mathcal E_i$, we deduce from Lemmas \ref{boundsolutions} and \ref{nosolutions} that $deg(\mathcal F, \mathcal E_i, -\lambda H) = 0$ for every $i$.
Using again formula \eqref{ED}, we obtain the decomposition $\mathcal D = \bigcup_{i \in \Z} \mathcal D_i$ and $deg(\mathcal G_{\epsilon_*}, \mathcal D_i, 0) = 0$.
Repeating the argument of the previous theorem for each $k$, $\mathcal A_k \subseteq \mathcal D_i$ where $i = j(\Omega_k, H)$, so if
\[\mathcal B_i = \mathcal D_i \setminus \bigcup_{k/ j(\Omega_k, H)=i} \mathcal A_k\]
then $deg(\mathcal G_{\epsilon_*}, \mathcal B_i, 0) = - \#\{k/ j(\Omega_k, H)=i\}$ (which might be eventually $0$, if there is no $k$ such that $j(\Omega_k, H)=i$).

We conclude that there exists one solution for each $\mathcal B_{j(\Omega_k, H)}$.

\end{proof}

As an example, in \cite{O1} we considered, using complex notation, $h(t) := e^{i t} + 27 e^{3 i t}$.
The function $H(t) = e^{i t} +3 e^{3 i t}$ is a parameterization of the epicycloid:

\begin{center}
	\includegraphics[bb=1 1 240 160, scale=.4]{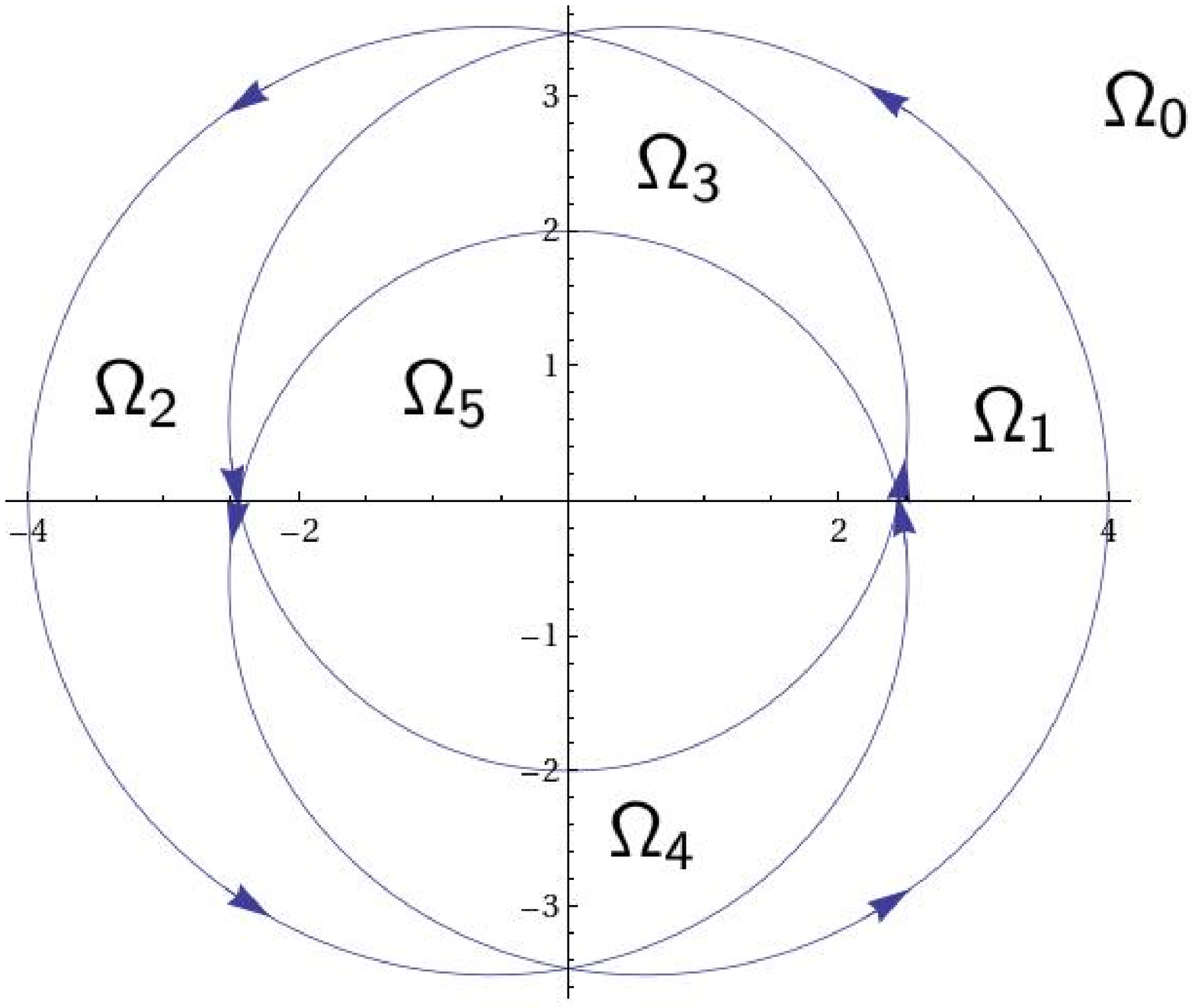}
\end{center}

As observed in \cite{O1}, $\R^2 \setminus H(\R)$ has five bounded connected components with corresponding indices $3,2,2,1,1$. 
Hence we obtained five periodic solutions.
But according to theorem \ref{mainthmrepimproved}, in the repulsive case the number of solutions is at least $8$.


In general, there is no way to guarantee that solutions of a given problem are non-degenerate without knowing them explicitly.
But since our functional is smooth, this property can be achieved by arbitrarily small perturbations. It is in some sense a `generic' property. This is the content of the next section.

\section{Genericity}
\label{genericity}
Some of the results of this section 
are proved in \cite{O2} in a more general situation. We include the proofs for the sake of completeness and clarity.

Let us define the spaces
\[C^i_{per} = \{x \in C^i(I, \R^n): x^{(j)}(0) = x^{(j)}(1), 0 \leq j < i\}\]
so $C^0_{per}$ is just $C^0 = C^0(I,\R^n)$.
Also, if $X$ is any of these spaces, we define

\[\tilde{X}= \{x \in X: \overline{x} = 0\}\]

We shall prove the following theorem:

\begin{theorem}
\label{mainthmgen}
There is a residual subset $\Sigma^2 \subseteq \tilde{C}^2_{per}$ with the following property:
if $O \subseteq C^2_{per}$ is an open bounded set such that $deg(\mathcal F,O,y) = n$ for some $y \in \Sigma^2$, 
then the equation $\mathcal F(x) = y$ has at least $n$ distinct solutions.
\end{theorem}

As a corollary, we shall obtain:
\begin{corollary}
In the conditions of 
Theorem \ref{mainthmrep}, there exists a residual subset $\Sigma^0\subset \tilde{C}^0_{per}$
such that \eqref{equa} has at least $2r$ solutions for $\lambda h \in \Sigma^0$ and $\lambda$ large enough.
\end{corollary}
It is interesting to compare this result with the fact, mentioned 
in \cite{AM}, that $\eqref{equa}$ has no solutions when $\|\lambda h\|_0 \leq \eta$ for some $\eta>0$ depending only on $q$. In particular,
the image of the operator $\mathcal F$ is not dense.

In order to prove our results, let us recall one of the main theorems about genericity in differential geometry.

\begin{theorem}{(Sard-Smale)}
Let $f:X \to Y$ be a Fredholm map between Banach manifolds. i.e, $f$ is $C^1$ and $df(x):T_x X \to T_{f(x)} Y$ is a Fredholm operator.
Then the set of regular values is a residual set in $Y$.
\end{theorem}

The main difficulty for applying the Sard-Smale theorem to our situation is the fact that we 
need to solve an equation of the form $\hat{F}(x) = -\lambda H$ with $H \in \tilde{C}^0$,
but $C^0 \setminus \tilde{C}^0$ is already a residual set. No information can be obtained from the theorem applied to $\hat{F}:C^2 \to C^0$.
Thus, we need to study a suitable restriction of the functional. 

Consider the operators
\[N:\mathcal U^i \subseteq C^i \to C^i,\;\; N(u)(t) = n(u(t)) := \frac {u(t)}{|u(t)|^{q+1}},\]
\[Q:C^i \to \R^n,\;\; Q(u) = \overline{u}\]
where $\mathcal U^i = \{x \in C^i:x(t) \neq 0, \;\; \forall t\}$.

\begin{lemma}
$N:\mathcal U^2 \to C^0$ is $C^1$. In particular, as $Q$ is linear and continuous, then
$QN$ is $C^1$.
\end{lemma}
\begin{proof}
Take $x \in \mathcal U^2$ and let $\eta:=d(Im(x),0)/2$. 
If $v \in T_x \mathcal U^2 = C^2$ satisfies $\|v\|_\infty\le \eta$, then 
\[n(x(t) + v(t)) = n(x(t)) + dn(x(t))(v(t)) + R(v(t)), \;\;\; R(x) = o(x).\]
In fact, we have $|R(x)| \leq C |x|^2$ with $C$ depending only on $\eta$.
Then
\[N(x+v) = N(x) + dN(x)(v) + R \circ v, \;\;\; \|R \circ v\|_\infty = |o(\|v\|_\infty)| \leq o(\|v\|_{C^2}) \]
and the proof follows.

\end{proof}

\begin{lemma}
The operator $QN:C^2_{per} \to \R^n$ has $0$ as a regular value.
In particular, the set $M := \{x \in C^2_{per}: QN(x) = 0\}$ is a differentiable manifold.
\end{lemma}
\begin{proof}
For each $w \in \R^n$ take $v(t) = dn(x(t))^{-1}(w)$, then
\[d(QN)(x)(v) = Q(dN(x)(v)) = \int_0^1 dn(x(t))(v(t)) dt = w\]
and hence $d(QN)(x)$ is an epimorphism for every $x$.
\end{proof}

Let us consider the operator 
$\hat D:{C}^2_{per} \to {C}^0_{per}$ defined by $\hat D(x):=x''$, so its restriction 
$D:\tilde{C}^2_{per} \to \tilde{C}^0_{per}$ is an isomorphism of  Banach spaces.
Let $\hat{F}: \mathcal U^2\cap C^2_{per} \to C^0_{per}$ be given by
\[\hat{F}(x) = D(x) - N(x)\]
and consider its restriction $F:M \to \tilde{C}^0_{per}$.

\begin{lemma}
\label{fredholmrestriction}
The operator $F$ is a Fredholm map of index $0$. 
\end{lemma}
\begin{proof}
Clearly $d\hat{F}(x) = D - dN(x)$ is the sum of a Fredholm map of index $0$ and a compact 
operator. Thus, $d\hat{F}(x)$ is a Fredholm linear operator.

The fact that $dF(x):T_x M \to \tilde{C}^0_{per}$ is Fredholm of index $0$ follows from the following 
general argument on vector spaces:

Let $\hat{L}:X \to Y$ be any linear Fredholm operator and let $L:V \to W$ be its restriction to spaces of finite codimension.  
Consider the following commutative diagram of Banach spaces 

\begin{displaymath}
	\xymatrix
	{
	0 \ar[r] & V \ar[r] \ar[d]^L & X \ar[r] \ar[d]^{\hat{L}} & X/V \ar[r] \ar[d]^{\iota} & 0\\
	0 \ar[r] & W \ar[r] & Y \ar[r] & Y/W \ar[r] & 0
	}
\end{displaymath}
where $\iota$ is canonical.
 
The long exact sequence given by the snake lemma has only finite dimensional spaces, which implies that $ind(L) - ind(\hat{L}) + ind(\iota) = 0$.

In our particular case, $X/V = Y/W = \R^n$ and $Q d_x \hat F = d_x QN$, so $\iota = id$.
\end{proof}

The following lemma is an immediate consequence of the above argument: 
\begin{lemma}
For $x \in M$, $d_x F$ is an isomorphism if and only if $d_x \hat{F}$ is an isomorphism.
\end{lemma}

Combining the previous results, we obtain:

\begin{lemma}
The set of regular values of $F$ is a residual set $\Sigma^0 \subseteq \tilde{C}^0_{per}$. Moreover:

\begin{enumerate}
 \item 
For $y \in \tilde{C}^0_{per}\;\; F^{-1}(y) = \hat{F}^{-1}(y)$ and regular values of $F$ are also regular values of $\hat{F}$.
\item
For each $y \in \Sigma^0$ and $x \in F^{-1}(y)$, $\hat{F}$ is a local diffeomorphism between neighborhoods of $x \in C^2_{per}$ and $y \in C^0_{per}$.
\end{enumerate}

\end{lemma}

Next, let us compose $\hat{F}$ with the isomorphism $id \oplus D^{-1}$ so we get

\[\mathcal F : C^2_{per} \to C^2_{per}, \mathcal F(x) = QNx \oplus D^{-1}PNx + Px.\]
This is the functional associated to equation \eqref{equa}. It is clear that $u$ is a solution of \eqref{equa} if and only if $\mathcal F(u) = -\lambda H$.
\medskip

Now we are in conditions to prove the main theorem of this section:
\medskip

\begin{potwr}{mainthmgen}
Take $y \in \Sigma^0 \subseteq \tilde{C}^0_{per}$. For every $x \in \mathcal F^{-1}(y)$, the function $\mathcal F$ is a 
diffeomorphism between neighborhoods of $x \in C^2_{per}$ and $D^{-1}(y) \in C^2_{per}$.
Taking $O_x$ a small neighborhood of $x$, from 
the product formula for the Leray-Schauder degree we obtain:
\[deg(\mathcal F,O_x, D^{-1}(y) ) = \pm 1.\]

Finally, take $\Sigma^2 = D^{-1}(\Sigma^0) \subseteq \tilde{C}^2_{per}$, which is residual since $D$ is an isomorphism. 
The result now follows by the excision property of the degree.
\end{potwr}

\section{Some examples}

\label{counterexamples}
\begin{prop}
For the repulsive case, if $Im(H)$ is contained 
in a line then the problem $Q_\epsilon$ has no solution for any $\epsilon$.
\end{prop}
\begin{proof}
Let us take coordinates $(x,y) \in \R \times \R^{n-1}$.
We may assume that the $y$ coordinate of $H$ is zero, that is
$H(t) = (H_1(t),0)$.

Let $z(t) = (x(t), y(t))$ be a solution. Multiply $y''$ by $y$ and integrate, then
\[-\int_0^1 |y'(t)|^2 dt = \int_0^1 \langle y(t), y''(t) \rangle dt = \int_0^1 \langle y(t), \frac {\epsilon y(t)} {|z(t) - H(t)|^{-(q+1)} } \rangle dt \geq 0, \]
so $y \equiv 0$.

Now, as $z$ lies in the same line as $H$, we have $H_1 > x$ or $ H_1 < x$ for all $t$ so either $x''$ is positive or negative.  
This contradicts the fact that $z$ is periodic.

\end{proof}

\begin{example}
The condition in Theorem \ref{oldthm} is not necessary.
\end{example}
Indeed, take $H_l(t) = e^{i l \sin(t)}$.
The curve $H_l \subseteq \R^2$ is degenerate for $l < \pi$, in the sense that $\R^2 \setminus Im(H_l)$ has no bounded connected components.

For $l = \pi$, we may construct as in \cite{O1} an open set $\Omega^*$ with $deg(F,\Omega^*,0) = 1$, where $F$ is the function defined in \eqref{funcF}.
Using the continuity of the Brouwer degree, we deduce that $deg(F,\Omega,0) = 1$ for some $l < \pi$ close to $\pi$. This provides 
a periodic solution of \eqref{equa}, although $\R^2 \setminus Im(H_l)$ has no bounded
connected components.

\begin{theorem}
Assume that $Im(H)$ is not contained in a line. 
Then for $\lambda$ large there exists a solution of the repulsive problem for some reparameterization of $H$.
\end{theorem}

\begin{proof}
Take $s_\epsilon:I \to I$ such that  $s_\epsilon$ is $C^1$ and increasing for $\epsilon > 0$, and $s_0$ is piecewise constant, $s_0([0,a]) = t_0, s_0((a,1]) = t_1$. Define $H_\epsilon(t) = H(s_\epsilon(t))$.
It is easy to see that $F$ has one non-degenerate zero $x_0$ of index $-1$. 
Choosing $a,t_0,t_1$ appropriately we can ensure that $x_0$ is not in $Im(H)$ (here we use the fact that $Im(H)$ is not contained in a line) 
and we take a small neighborhood $U$ of $x_0$ not touching $H$.
We obtain a zero $x_1 \in U$ of $F$ for small values of $\epsilon$ as a non-degenerate zero of index $-1$, which is isolated by $U$.
This point is continued to a small solution of \eqref{equa} for $\lambda$ large enough and such that
\[deg(\mathcal F_\lambda,\mathcal U,0) = -1\]
where
\[\mathcal U = \{x \in C^2_{per}/ Im(x) \subseteq U\}.\]
\end{proof}

\section{Higher dimensions}
\label{higherdimensions}
Theorem \ref{oldthm} and the extensions Theorem \ref{mainthmrep} and Theorem \ref{mainthmgen} can be carried out in dimension 
$n > 2$ without any modification of the proofs, as far as 
we can find open sets $\Omega \in \R^n \setminus Im(H)$ such that the degree over $\Omega$ of the map $F: \R^n \setminus{Im(H)} \to \R^n$ defined in the introduction is well defined and different from zero.

In this section we shall construct open sets with this property.
We make use of singular homology theory with coefficients in a field to obtain information about the degree of $F$.

\subsection{Dimension 3}

For convenience, let us define the function $G:\R^n \setminus Im(H) \to \R$ given by 
\[G(x) := \int_0^1 g(x-H(t))dt= \int_0^1 \frac 1{|x-H(t)|^{q-1}}dt.\]

\begin{theorem}
\label{mainthmR3}
If $\R^3 \setminus Im(H)$ is not simply connected then there exists an open set $\Omega \subseteq \R^3 \setminus Im(H)$ such that $deg(\nabla G,\Omega,0) \neq 0$.
Moreover, if $r:= dim(H_1(\R^3 \setminus Im(H)))$ then $r \neq 0$ and $deg(\nabla G, \Omega, 0) \geq r$.
\end{theorem}

\begin{remark}
The number $r$ counts the self-intersections of $H$. In the case that $Im(H)$ is contained in a plane $P$, 
it is exactly the number of connected components of $P \setminus Im(H)$. In this way, we recover Theorem \ref{oldthm}, although 
not in its full generality.

Also, it is worth noticing that
the fundamental group of $\R^3 \setminus Im(H)$ distinguishes whether $H$ is or not a non-trivial knot, but the homology does not. 
In fact, using Alexander duality (Lemma \ref{duality}), one can show that in most cases $r$ depends on the image of $H$ and not on how this image is embedded in $\R^3$.
\end{remark}

Our proof of Theorem \ref{mainthmR3}
will require several lemmas;  all of them shall be stated in $\R^n$.

\begin{lemma}
\label{smooth}
The function $G$ is $C^\infty$ smooth in $\R^n \setminus Im(H)$.
\end{lemma}

\begin{lemma}
\label{coercive}
${}$

\begin{itemize}
\item $G(x) \to \infty$ when $x \to x_0 \in H$.

\item $G(x) \to 0$ when $|x| \to \infty$.

\item $G(x) < g(d(x, H))$

\end{itemize}
\end{lemma}

\begin{lemma}
\label{subharmonic}
The function $G$ is sub-harmonic for $q > n - 1$ and harmonic
for $q = n - 1$. In particular, it has no local maxima in the interior of the domain of definition.
In consequence, if $U \subseteq \R^n$ is an open and bounded set such that $H \cap \overline{U} = \emptyset$ then it attains its maximum at the boundary.
\end{lemma}
\begin{proof}
It follows directly 
from the fact that $g$ is 
sub-harmonic for $q>n-1$ 
and harmonic for $q = n - 1$, and that $\Delta G = \int_0^1 \Delta g(x-H(t))dt$.
\end{proof}

%
%
\begin{lemma}
\label{faraway}
If $B$ is a large ball centered at the origin then $-\nabla G$ is homotopic to the identity in $\partial B$.
\end{lemma}

\begin{proof}
As
\[dG(x)(x) = \int_0^1 dg(x-H(t))(x) dt\]
and
\[dg(x-H(t))(x) = -(q-1) |x-H(t)|^{-(q+1)} (\langle x, x \rangle - \langle H(t), x \rangle ),\]
then for $|x| > \|H\|_\infty$ 
the value of $\langle \nabla G(x), x\rangle$ is negative.

\end{proof}

Now some general lemmas

\begin{lemma}
\label{connected}
For $n\geq 3$ the set $\R^n \setminus Im(H)$ is arcwise-connected.
\end{lemma}
\begin{proof}
It follows as an application of transversality.
\end{proof}

\begin{lemma}{(Alexander duality, see \cite[p. 296, thm 16]{SP})}
\label{duality}

Let $U$ be an open set in $S^n$ with smooth boundary. If $k_*$ and $e_*$ denote the reduced Betti numbers of $U$ and $S^n \setminus U$ respectively then
\[k_q = e_{n-1-q}.\]
\end{lemma}

For convenience, if $U\subset \R^n$ is open and bounded and 
$\phi:\partial U\to \R^n\setminus\{ 0\}$ is continuous, 
we shall use the notation $deg(\phi, \partial U,0):=
deg(\hat\phi, U,0)$ where $\hat \phi:\overline U\to \R^n$ is any continuous extension of $\phi$.

\begin{lemma}{(Hopf, see \cite[Satz VI]{HO})}

\label{normaldegree}

Let $U$ be an open bounded set in $\R^n$ with smooth boundary.\\ Let $n:\partial U \to \R^n$ be the outer-pointing unit normal vector field.
Then $deg(n,\partial U,0) = \chi(U)$, where $\chi$ denotes the Euler characteristic.
\end{lemma}

\begin{potwr}{mainthmR3}
We will construct an open set $U$ with the following properties:
\begin{itemize}
\item $\partial U$ is smooth.
\item $U \supseteq H$.
\item $G$ is constant and $\nabla G \neq 0$ in $\partial U$.
\item $\chi(U) \le 1-r$ where $r := dim(H^1(\R^3 \setminus Im(H)))$.
\end{itemize}
Once we have this set $U$, we may notice that $-\nabla G$ is orthogonal to $\partial U$, so $deg(-\nabla G, \partial U, 0) = deg(n, \partial U, 0)$ where $n$ is the outer-pointing unit normal vector field.

Then using \eqref{normaldegree} it follows that $deg(-\nabla G, \partial U, 0) = \chi(U) \leq 1 - r$.
Next, take a large ball $B$ given by Lemma \ref{faraway} and observe that $deg(-\nabla G, \partial B, 0) = 1$. 
Finally, as $Im(H) \subseteq U$, it follows that $G$ is well defined in $\Omega = B \setminus U$ and $deg(-\nabla G, \Omega, 0) \geq r$.

Now we may construct $U$ as follows. 
For convenience, we shall regard $\R^3$ as embedded in $S^3$, and call $N \in S^3$ the north pole. 
We remark that the function $G$ extends continuously to $N$ by setting $G(N) = 0$.

It follows from the hypothesis that there exist smooth curves $\gamma_1 \ldots \gamma_r \in \R^3 \setminus Im(H)$ generating $H_1(\R^3 \setminus Im(H))$ as a basis.
By Lemma \ref{connected}, each $\gamma_i$ is connected to $N$ by another curve $\delta_i$. 
Since $G$ is continuous in $\bigcup \gamma_i \cup \delta_i$, then it is bounded by a number $\alpha_0$.  By Sard's lemma, there exists a regular value $\alpha > \alpha_0$.

Next, take $V$ the connected component of $\{G < \alpha\}$ that contains $N$ and let $U = \overline V^c$. $U$ and $V$ are manifolds with 
common boundary in $S^3$. Moreover, $U$ has a finite number of connected components that are disjoint manifolds with boundary. Obviously 
$Im(H)\subseteq U$.

We claim that $U$ is in fact connected.
Indeed, let $U'$ be a connected component of $U$. Then $U' \cap Im(H) \neq \emptyset$: otherwise, from Lemma 
\ref{subharmonic} we deduce that $G|_{U'}$ attains its maximum at some $x_0 \in \partial U' \subseteq \partial U$ so $G(x_0) = \alpha$.
But also $G\leq \alpha$ in $V$, so $x_0$ is a local maximum of $G$ in $\R^3 \setminus Im(H)$, a contradiction since $\alpha$ is a regular value.
As $Im(H) \subseteq U$ is connected, we conclude that $U$ is connected.

Since $G = \alpha$ in $\partial U$ and $G < \alpha$ in $\bigcup \gamma_i \cup \delta_i $, it follows that $\gamma_i \subseteq V$ and that $U$ is bounded.

Notice that there is a homomorphism induced by the inclusion 
$H_1(V) \to H_1(\R^3 \setminus Im(H))$, which sends $[\gamma_i]_{H_1(V)}$ to the generators, so it is surjective and hence $dim(H_1(V)) \geq r$.
Using the notation of Lemma \ref{duality}, $e_1 \geq r$.

Using the Alexander duality (Lemma \ref{duality}), from the fact that $V$ is connected we obtain $k_2 = e_0 = 0$. Also, $k_0 = 0$ because $U$ is connected, and
$k_1 = e_1 \geq r$.
The Euler characteristic of $U$ is computed then as $\chi(U) = 1 + k_0 - k_1 + k_2 \leq 1-r$, and so completes the proof. 
\end{potwr}

\subsection{Dimension \texorpdfstring{$n > 3$}{n > 3}}
In this section we shall prove, for $n>3$, the existence of a set $U$ as before.

\begin{theorem}
\label{mainthmRn}
If $H$ is an embedding then there exists an open set $U \subseteq \R^n$ such that $H \subseteq U$ and $deg(\nabla G,U,0) = 0$.
\end{theorem}

As a consequence, we may construct an open set $\Omega$ such that
\[deg(-\nabla G, \Omega, 0) = 1.\]

\begin{lemma}
\label{accumulate}
If $H$ is an embedding and $q > 1$ then the critical points of $G$ do not accumulate in $H$.
\end{lemma}
\begin{proof}
Assume, by contradiction, that there is a sequence $x_n$ of critical points of $G$ accumulating somewhere in $H$.
Without loss of generality we may assume $x_n \to x_0 \in Im(H)$. Fix $t_n \in I$ such that the distance from $x_n$ to $Im(H)$ is realized in $H(t_n)$.
Let $v_n = x_n - H(t_n)$ and $\lambda_n = |v_n|^{-1}$. Again without loss of generality, we may assume $\lambda_n v_n \to y \in \R^n$ and $t_n \to t_0$.
By periodicity we may also assume $t_0 = 0$.
Let us compute $dG(x_n)$:
\[0 = dG(x_n) = \int_0^1 dg(x_n - H(t)) dt = \lambda_n^q \int_0^1 dg(\lambda_n x_n - \lambda_n H(t)) dt\]
\[= \lambda_n^{q-1} \int_0^{\lambda_n} dg(\lambda_n x_n - \lambda_n H(\frac s {\lambda_n} + t_n)) ds\]
\[ = \lambda_n^{q-1} \int_0^{\lambda_n} dg(\lambda_n (x_n - H(t_n)) - \lambda_n (H(\frac s {\lambda_n} + t_n) - H(t_n)) ds.\]
It follows from the assumptions and the regularity of $H$ that $$\lambda_n (x_n - H(t_n)) \to y\qquad \hbox{and }\quad {\lambda_n (H(\frac s {\lambda_n} + t_n) - H(t_n)) \to s H'(t_0)}$$ for every $s$.
In order to establish the convergence of the integral, let us estimate $\lambda_n x_n - \lambda_n H(\frac s {\lambda_n} + t_n)$. 

Consider the continuous function
\[
\gamma(t,s) := \left\{
\begin{array}{c}
	\frac {|H(t+s) - H(t)|}{|s|} \hbox{ if } t \neq s\\
	|H'(t)| \hbox{ if } t = s
\end{array}
\right.
\]

As $H$ is an embedding, $\gamma(s,t) > 0$ for all $s$ and $t$ then by compactness $|H(t+s) - H(t)| \geq \epsilon |s|$ for all $s, t$ and some small $\epsilon > 0$.

Using this we obtain
\[\left|\lambda_n x_n - \lambda_n H\left(\frac s {\lambda_n} + t_n\right)\right| \geq \lambda_n \left(\left|H\left(\frac s {\lambda_n} + t_n\right) - H(t_n)\right| - |v_n|\right) \geq \epsilon |s| - 1.\]
As $H(t_n)$ is the nearest point to $x_n$, the left hand side is also bounded from below by $\lambda_n |v_n| = 1$ so
\[|dg(\lambda_n x_n - \lambda_n H(\frac s {\lambda_n} + t_n))| \leq \max\{1, (\epsilon |s| - 1)^{-q}\}(q-1)\]
which is integrable in $\R$ for $q>1$.

By dominated convergence we conclude that $\lambda_n^{-(q-1)} dG(x_n)$ converges to $\int_0^{\infty} dg(y - s H'(t_0)) ds$.
Notice that $\langle H'(t_n), \lambda_n v_n \rangle = 0$ and, taking limits, $\langle H'(t_0), y \rangle = 0$ so the integral can be 
explicitly calculated and is different from zero, a contradiction.
\end{proof}

\begin{potwr}{mainthmRn}
Using Lemma \ref{coercive} we can take $\alpha > 0$ large enough so that $U:=\{G > \alpha\}$ is close to $H$. By Lemma \ref{accumulate} we can ensure that there are no critical points of $G$ in $U$.

Applying the Morse deformation lemma to $G$ at level $+\infty$ we deduce that $U$ is a deformation retract of $H$ so $deg(-\nabla G, \partial U,0) = \chi(U) = \chi(H) = \chi(S^1) = 0$.

\end{potwr}

\begin{remark}
As the set of embeddings of $S^1$ in $\R^n$ for $n \geq 3$ is open in the $C^1$ topology and dense in the $C^\infty$ topology, Theorem \ref{mainthmRn} is indeed a genericity result.
\end{remark}

\section{On Morse functions and knots}

\label{morseandknots}
In this section we prove that, generically, the function $G$ has some differential structure that allows us to 
obtain more solutions when the curve $H$ is knotted. The results are stated in $\R^3$.

Let us recall the parametric version of the Sard-Smale's theorem. Firstly, we need the following

\begin{definition}
Let $f:X \to Y \supset W$ be a smooth function between Banach manifolds and $W$ a submanifold.
We shall say that $f$ is transversal to $W$ and denote $f \trans W$ if for each $x \in f^{-1}(W)$ we have that the composition
\[T_x X \to T_{f(x)} Y \to T_{f(x)} Y / T_{f(x)} W\]
is a submersion (namely, it is surjective and its kernel splits).
\end{definition}

\begin{theorem} (see \cite{AMR})

Let $f:X \times B \to Y \supset W$ be a smooth function between Banach manifolds and $W$ a submanifold. Assume that $X$ is finite dimensional and $W$ is finite codimensional.

If $f \trans W$ then $f_b \trans W$ for $b$ in a residual subset of $B$, where $f_b$ is the function $x \mapsto f(x,b)$.

\end{theorem}
%
%

We call $E = \R^n$ the Euclidean space and consider the function
\[ DG:C^2_{per}(I, \R^n) \times E \to E^*\]
\[DG(H, x) = \int_0^1 dg(x - H(t)) dt\]
where, as before, $g(x) = \frac 1{|x|^{q-1}}$.

\begin{lemma}
The function $DG$ is transversal to $0 \in E^*$.
\end{lemma}
\begin{proof}
Notice that $d^2g(x) \in L(E, E^*)$ is always invertible.
For given $\gamma \in E^*, H\in C^2, x \in E$ take $\hat{H}(t) = (d^2g(x-H(t)))^{-1}(\gamma)$, $\hat{H} \in C^2(S^1,\R^n)$. 
Then $d(DG)(H,x)(\hat{H},0) = -2 \pi \gamma$. Since $Ker(d(DG))$ is finite codimensional, it splits. Then $DG$ is transversal to $0$ (notice that $DG^{-1}(0)$ is a differentiable Banach manifold). 
\end{proof}

In order to  emphasize the dependence on $H$, we shall use the notation
$G_H(x):= \int_0^1 g(x - H(t)) dt$ 
\begin{lemma}

$DG_H$ is transversal to $0$ if and only if $G_H$ is a Morse function.
\end{lemma}
\begin{proof}
Note that $d G_H (x) = DG(H,x)$ so $DG_H$ is transversal to $0 \in E^*$ if an only if $d G_H$ is transversal to $0$, if and only if for each critical point $x$ of $G_H$, $d^2_x G_H = d_x(d G_H ) \in L(E, E^*) $ is invertible. This $G_H$ is a Morse function.
\end{proof}

We conclude that there exists a residual set $\Sigma \subseteq C^2_{per}(I, \R^n)$ such that 
if $H\in \Sigma$ then $G_H$ is a Morse function.

\begin{lemma}
Let $f:U \to R$ be a Morse function and let $x_0$ be a critical point of index $\lambda$.
Then there exists a neighborhood $V$ that isolates $x_0$ as a critical point, and such that

\[deg(\nabla f, V, 0) = (-1)^\lambda.\]
\end{lemma}
\begin{proof}
Using the Morse lemma we obtain a chart
$(\hat V,x_1, \ldots, x_n)$ such that $f(x_1, \ldots, x_n) = - \sum_{i < \lambda} x_i^2 + \sum_{i \geq \lambda} x_i^2$. Take a square inside $\hat V$ of the form 
\[V:=\{(x_1, \ldots, x_n): \|(x_1, \ldots, x_{\lambda})\| \leq \epsilon, \|(x_{\lambda+1}, \ldots, x_n)\| \leq \epsilon \}.\]
This is the required neighborhood.
\end{proof}

Summarizing, when $H \in \Sigma$
we obtain at least one solution of \eqref{equa} for each critical point of $G_H$.

\begin{remark}
In the present case, in which 
$G$ is a Morse function, the use of Leray-Schauder degree may be avoided by simply invoking the implicit function theorem for Banach spaces.
\end{remark}



The following theorems give us relations between 
the number of critical points of $G$ and some knot invariants of $H$ (see \cite{AD} for 
a general overview on knot theory).
Although its statement 
is contained in that of Theorem \ref{tunnelthm}, we shall give an independent 
proof because it is much simpler, it uses different tools and 
part of it shall be used later.

\begin{theorem}
\label{trivialknotthm}
Assume $H$ is a non-trivial knot embedded in $\R^3$ and that $G$ is a Morse function.
Then $G$ has at least $3$ critical points in $\R^3$.
\end{theorem}

\begin{proof}
Changing the function $G$ near $Im(H)$ we may assume that $-G$ is a Morse function in $S^3$ with a global minimum $m \in Im(H)$ and an index $1$ critical point in $Im(H)$.
Consider the Morse decomposition of $S^3$ through the function $-G$.
Call $M^c = \{x/-G(x) \leq c \}$ and $M^{-\infty} = Im(H) = S^1$.

We have a cell complex given by the Morse decomposition of $S^3$.
\[0 \to H_3(X_3, X_2) \to H_2(X_2, X_1) \to H_1(X_1, X_0) \to H_0(X_0) \to 0\]
\[0 \to \Z^{m_3 }\oplus N.\Z \to \Z^{m_2} \to \Z^{m_1} \oplus H.\Z \to \Z^{m_0} \oplus m.\Z \to 0\]
with homology $H_*(S^3) = \{\Z, 0, 0, \Z\}$.

Here $m_i$ is the number of critical points of $-G$ of index $i$ in $\R^3 \setminus Im(H)$ and $N$ is the north pole, where we have a global maximum.
Since $G$ is subharmonic we have $m_0 = 0$.

Computing the Euler characteristic we have
\[\chi(S^3) = \chi(Im(H)) + m_0 - m_1 + m_2 - (m_3 + 1)\]
\[0 = 0 - m_1 + m_2 - (m_3+1)\]
\begin{equation}
\label{m1m2m3}
m_2 - m_1 = m_3 + 1 \geq 1.
\end{equation}

We see that the only way to have just $1$ critical point in $\R^3$ is $m_2 = 1, m_1 = 0, m_3 = 0$. All other possibilities give $3$ or more. Then the cell complex reduces to
\[0 \to N.\Z \xrightarrow{\; 0 \;} e.\Z \to H.\Z \xrightarrow{\; 0 \;} m.\Z \to 0\]

We know that $d(N) = 0$ because $H^3(S^3) = \Z$, and $d(H) = m - m = 0$.

Now consider the attaching map $f:S^1 = \partial e \to X_1$ corresponding to the unique $2$-cell $e$.

Since $d: e.\Z \to H.\Z$ is an isomorphism we have $d(e) = \pm H$.

Now we may consider $f:D^2 \to S^3$ as the inclusion of the $2$-cell. Composing with the isotopy 
generated by the negative gradient, we may assume that $G(f(S^1))$ is uniformly large so $Im(f)$ is uniformly close to $Im(H)$ and thus lies in a tubular neighborhood of $Im(H)$. 
It is clear that $f$ is homotopic to $H$ or $H^{-1}$ inside that tubular neighborhood.

It follows that $f(S^1)$ is a satellite knot whose companion is $H$. 
From the genus formula for satellite knots, since $f$ is homotopically nontrivial in the tubular neighbourhood, we see that $f(S^1)$ is a nontrivial knot.
But this is a contradiction, since it is the boundary of an embedded $2$-cell in $\R^3$. 
\end{proof}

%

\begin{remark}
A lower bound for the number of solutions 
may be also obtained by considering the presentation 
of the knot group given by the Morse decomposition of the knot complement.
Namely, a Morse function in $S^3 \setminus Im(H)$ gives a presentation of $\pi_1(S^3 \setminus Im(H))$ with one generator for each critical point of index $1$ and one relation for each critical point of index $2$. 
Thus an obvious lower bound for the number of critical points would be the minimal numbers of generators (and relations) of a presentation of the group.
\end{remark}

A knot invariant that seems to be closely related to the minimal number of critical points of $G$ is the following:

\begin{definition}
Let $K \subseteq S^3$ be a knot. A tunnel is an embedded arc with endpoints in $K$.
The tunnel number of $K$, $t(K)$ is the minimal number of tunnels $t_i$ such that if $N = N(K, t_i)$ is a tubular neighbourhood of the knots and the tunnels, then $S^3 \setminus N$ is a handlebody. We call this set of tunnels a tunnel decomposition.
\end{definition}

The only tunnel number zero knot is the trivial knot, so Theorem \ref{trivialknotthm} is a particular case of our next result. 

\begin{theorem}
\label{tunnelthm}
Assume $H$ is a knot embedded in $\R^3$ and that $G$ is a Morse function.
Then $G$ has at least $2 t(H) + 1$ critical points and consequently equation \eqref{equa} has at least $2 t(H) + 1$ periodic solutions for $\lambda$ large enough.
\end{theorem}

\begin{remark}
If a knot has tunnel number one, it will be a one-relator knot and therefore prime \cite{NO}.
Then for any composite knot we will have at least $5$ critical points.
Also, it is worth noting that a decomposition with $t$ tunnels gives a presentation of $\pi_1(S^3 \setminus Im(H))$ with $t$ relations so the present theorem provides a better bound for the number of solutions than the one stated in the previous remark.
\end{remark}

\begin{potwr}{tunnelthm}
By the Kupka-Smale theorem \cite[pp 159, thm 6.6]{AB}, we may perturb $G$ (preserving the critical points and its indices), in order to obtain a Morse-Smale function.

For each critical point $p$ of $-G$ of index $1$, denote $\gamma_p:\R \to S^3$ the embedding of the unstable manifold. 

Clearly $\gamma_p$ is connected with $Im(H)$ because $G$ is Morse-Smale.
Now take a tubular neighborhood $U$ of $Im(H) \cup Im(\gamma_{p_1}) \cup \ldots \cup Im(\gamma_{p_k})$. We will show that $S^3 \setminus U$ is a handlebody then $k$, the number of $\gamma$ curves, is greater than or equal to $t(H)$.

For each critical point $q$ of $-G$ of index $2$, denote $\delta_q:\R \to S^3$ the embedding of the stable manifold.
Again, $\delta_q$ connects with $N \in S^3$ because $G$ is Morse-Smale. Obviously, $\gamma_{p_i}, \delta_{q_j}$ are  disjoint.

We shall construct a tubular neighborhood $V$ of $\{N\} \cup Im(\delta_{q_1}) \cup \ldots \cup Im(\delta_{q_s})$. $V$ is a handlebody and we will show that $S^3 \setminus U$ retracts to $V$.

Consider the positive gradient flow of $-G$ (i.e. $-\nabla G$ ) denoted by $\phi^t(x)$.
Take a point $x \in S^3 \setminus U$.
If the orbit of $x$ converges to a critical point $q$ then it belongs to the unstable manifold of $-G$ at $q$.
We know that $x$ cannot belong to any $\gamma_{p_i}$ so $q$ is a critical point of $-G$ of index $2$ or $q = N$.
We deduce that $q \in V$. 

Orbits always converge to critical points, then every point in $(S^3 \setminus U) \setminus V$ must eventually enter at $V$. 
If we manage to construct $V$ to be transversal to the flow, 
then, by the inverse function theorem, the arrival time at $V$ is a continuous function $t(x)$. This allows us to construct 
the required deformation as
\[\psi(t,x) = \phi^{\min\{t,t(x)\}}(x).\]

Finally, using formula \eqref{m1m2m3} we obtain at least $t(H)+1$ critical points of index $2$.

Now we shall construct $V$.
With this purpose, take a critical point $q$ of index $2$ and a Morse like system of coordinates $x_i$ in a neighborhood of $q$ of the form $U = [-1,1] \times D^2$ where $\{\pm 1\} \times D^2$ is the exit set.
We know that $\delta_q$ converges to the attractor $N$.

Consider
\[\alpha^\pm: \R_+ \times D^2 \to S^3\]
\[\alpha^\pm(t,x) = \phi^t(\pm 1,x)\]

Let us show that $\alpha^\pm$ is an embedding.

Clearly $t \mapsto \phi^t(x)$ is injective because $f$ is strictly decreasing in the integral curves.
Moreover, points leaving $U$ can never return since $f$ restricted to the exit set is less or equal to $f$ restricted to the entrance set.
We conclude that the orbits of the points in $\{\pm 1\} \times D^2$ are all different and then $\alpha^\pm$ is injective.
It is clear, also, that the image is open and $\alpha$ is an embedding.

Notice that $\alpha$ sends $\R_+ \times \{x\}$ to integral lines.
Now we may easily construct a neighborhood $V_q$ of $\R_+ \times \{0\}$ with boundary transversal to the horizontal flow $t,x \mapsto t+s,x$.
Call $F = V_q \cap \{0\} \times D^2$.
It follows that $\alpha^+(V_q) \cup ([-1,1] \times F) \cup \alpha^-(V_q)$ is a neighborhood of $\delta_q$.
Taking the union of these neighborhoods and an attracting neighborhood of $N$ we obtain the set $V$ with the desired properties.
\end{potwr}

\section{Links and the restricted \texorpdfstring{$n$}{n}-body problem}
\label{restrictednbody}

Consider the equation

\begin{equation}
\label{restrnbodyequa}
\left\{
\begin{array}{cc}
z''(t) = \sum_{i=1}^{n-1} \nabla g(z - \lambda p_i(t)) + \lambda h(t)\\
z(0) = z(1), \;\; z'(0) = z'(1)
\end{array}
\right.
\end{equation}
for $z \in \R^3$, where $p_i:I \to \R^3$ are arbitrary periodic functions.

This equation describes the motion of a particle $z$ under the force of the gravitational attraction of $n-1$ bodies moving along large periodic trajectories $\lambda p_i(t)$, and under an arbitrary force $\lambda h(t)$ of comparable intensity.
The letter $n$ stands for the number of bodies and not for the dimension, which is now equal to $3$.

As before, assume $\overline{h}=0$ so we are able to make the change of variables
\[z = \lambda (x - H(t))\]
where $H$ is a periodic second primitive of $-h$. 

With the new variables, the equation is transformed into:
\begin{equation}
\left\{
\begin{array}{cc}
x''(t) = \epsilon \sum_{i=1}^{n-1} \nabla g(x -(H(t) + p_i(t)))\\
x(0) = x(1), \;\; x'(0) = x'(1).
\end{array}
\right.
\end{equation}

Let $G_t:W \to \R$ be given by
\[G_t(x) = \sum g(x - (H(t) + p_i(t))),\]
then clearly
\[\nabla G_t(x) = \sum \nabla g(x - (H(t) + p_i(t)))\]
so the equation becomes
\begin{equation}
\left\{
\begin{array}{cc}
x''(t) = \epsilon \nabla G_t(x)\\
x(0) = x(1), \;\; x'(0) = x'(1).
\end{array}
\right.
\end{equation}
Thus, solutions of
(\ref{restrnbodyequa})
are related to the function $G(x) = \int_0^1 G_t(x) dt$ in the same manner as in the previous sections.

The set in which $G= \infty$ is the union of the curves $k_i = H+p_i$, so we define $K = \bigcup_{i=1}^{n-1} Im(k_i)$.

Theorem \ref{mainthmR3} may easily be generalized as follows:

\begin{theorem}
\label{mainthmR3gen}
Let $r = dim(H^1(\R^3 \setminus K))$. Then there exists an open set $\Omega \subseteq \R^3 \setminus K$ such that $deg(\nabla G, \Omega, 0) \geq r-n+2$.
\end{theorem}
\begin{proof}
The proof follows exactly as in the proof of Theorem \ref{mainthmR3} except that now the set $U$ is not connected.
From the discussion in the mentioned proof, 
it follows that every connected component 
of $U$ touches $K$ and thus $U$ has at most $n-1$ connected components.
We deduce that $\chi(U) \leq n - 1 - r$ and so completes the proof.
\end{proof}
\begin{remark}
Notice that the trajectories of the curves $k_i$ may have self-intersections and intersect each other.
In the statement of the preceding theorem, the number 
$r-n+2$ may be replaced by $r-c+1$, where $c$ is the number of connected components of $K$.
\end{remark}

Theorem \ref{tunnelthm} may be generalized as follows:

\begin{theorem}
\label{tunnelthmgen}
Assume $K$ is a link embedded in $\R^3$ and that $G$ is a Morse function.
Then $G$ has at least $2 t(K) + 1$ critical points.
\end{theorem}

\begin{remark}
It is easy to see that $G$ is a Morse function for `generic' $K$.
\end{remark}

The proof requires no modification at all with respect to that of Theorem \ref{trivialknotthm}, but some explanation about formula \eqref{m1m2m3} 
is needed. 
Indeed, 
now we have to represent each strand of the link $K$ by one $0$-cell and one $1$-cell attached. The cell complex has to be replaced by
\[0 \to \Z^{m_3 }\oplus N.\Z \to \Z^{m_2} \to \Z^{m_1} \oplus \Z^{n-1} \to \Z^{m_0} \oplus \Z^{n-1} \to 0\]
so formula \eqref{m1m2m3} still holds.

We are already in conditions to establish the main results of this section

\begin{theorem}
Define $r$ as in Theorem \ref{mainthmR3gen}. If $r \geq n-1$, then for $\lambda$ large enough there exist at leasts 
one periodic solution of \eqref{restrnbodyequa}, and generically 
$r-n+2$ distinct solutions.
\end{theorem}

\begin{theorem}
Assume that $K$ is a link embedded in $\R^3$ and that $G$ is a Morse function.
Then equation \eqref{restrnbodyequa} has at least $2 t(K) + 1$ distinct solutions for $\lambda$ large enough.
In particular, it must have at least $2n - 3$ distinct solutions even if $K$ is the unlink.
\end{theorem}

\end{document}